\documentclass[12pt]{article}
\usepackage{}
\usepackage{mathrsfs}
\usepackage{amsmath}
\usepackage{amsmath,amsthm,amssymb,amscd}
\usepackage{latexsym}
\usepackage[colorlinks,linkcolor=blue,anchorcolor=blue,citecolor=blue,CJKbookmarks=True]{hyperref}
\usepackage[numbers,sort&compress]{natbib}
\usepackage{hypernat}
\usepackage{cases}
\usepackage{natbib}
\usepackage{geometry}
\usepackage[all,pdf]{xy}

\allowdisplaybreaks

\newtheorem{theorem}{Theorem}[section]
\newtheorem{corollary}[theorem]{Corollary}
\newtheorem{lemma}[theorem]{Lemma}
\newtheorem{proposition}[theorem]{Proposition}

\theoremstyle{definition}

\numberwithin{equation}{section}

\begin{document}


\baselineskip=17pt

\title{{\bfseries Annihilator ideals of indecomposable modules of finite-dimensional pointed Hopf algebras of rank one}}

\author{
   {\small Yu Wang}\\
   {\small  School of Mathematics and Physics,
   Jiangsu University of Technology,}\\{\small
Changzhou, Jiangsu 213001, P. R. China}\\
{\small Email address: yu.wang@mail.sdu.edu.cn}
}

\date{}

\maketitle

\renewcommand{\thefootnote}{}
\footnote{}
\renewcommand{\thefootnote}{\arabic{footnote}}
\setcounter{footnote}{0}
\begin{abstract}
Let $H$ be a finite-dimensional pointed Hopf algebra of rank one over an algebraically closed field of characteristic zero. In this paper we describe all annihilator ideals of indecomposable $H$-modules by generators. In particular, we give the classification of all ideals of finite-dimensional pointed Hopf algebra of rank one of nilpotent type over Klein $4$-group.

\end{abstract}

\section{Introduction}

\indent

We start with a brief introduction to the history of pointed Hopf algebras of rank one. Let $A$ be a Hopf algebra over a field $\Bbbk$ and $A_0\subseteq A_1\subseteq A_2\subseteq\cdots$ the coradical filtration of $A$. Assume that $A_0$ is a Hopf subalgebra of $A$. If $A$ is generated by $A_1$ as an algebra and $\dim_{\Bbbk}(\Bbbk\otimes_{A_0}A_1)=n+1$, then $A$ is called a Hopf algebra of rank $n$. If every simple subcoalgebra of $A$ is one-dimensional, then $A$ is called pointed. Krop and Radford \cite{MR2236601} classified all finite-dimensional pointed Hopf algebras of rank one over an algebraically closed field of characteristic zero in terms of group data. And a characterization of finite-dimensional pointed Hopf algebras of rank one over an algebraically closed field of prime characteristic was obtained by Scherotzke in \cite{MR2397414}. He gave three types of Hopf algebras presented by generators and relations, and determined the indecomposable projective modules for certain Hopf algebras in every type. Wang et al. \cite{MR3357949} studied the representation theory of arbitrary dimensional pointed Hopf algebras of rank one over an arbitrary field. And they described the structure of all simple weight modules and all finite-dimensional indecomposable weight modules. Wang et al. \cite{MR3284336,MR3448167} studied the Green ring $r(H)$ of $H$, where $H$ is a finite-dimensional pointed Hopf algebras of rank one over an algebraically closed field of characteristic zero, and presented it in terms of generators and relations. They also studied the Jacobson radical $J(r(H))$ of the Green ring $r(H)$ and the Green ring of the stable category of left $H$-modules.

Recall that a ring is called a principal ideal ring if every two-sided ideal is generated by a single element. Principle ideal rings are important objects in ring theory and have many applications in representation theory. It is of great interest to determine whether a ring is a principal ideal ring. There has been momentous attention to principal ideal rings in history. For instance, Fisher et al. \cite{MR0049909,MR0396653,MR0457540} discussed that under certain conditions a group algebra over an arbitrary field is a principal ideal ring. Afterwards, Decruyenaere and Jespers \cite{MR1314687} investigated when a commutative ring with identity graded by an abelian group is a principal ideal ring.  Chimal-Dzul and L\'{o}pez-Andrade \cite{MR3775984} studied the polynomial ring $R[x]$, where $R$ is a finite commutative ring with identity. They proved that $R[x]$ is a principal ideal ring if and only if $R$ is a finite direct product of finite fields. Lately Alvarado-Garc\'{i}a et al. \cite{MR3928507} introduced the concepts of parainjectivity and paraprojectivity, and they obtain some characterizations of artinian principal ideal rings.

There has been also substantial interest in semigroup algebras. Decruyenaere et al. \cite{MR1125068} characterized commutative semigroup algebras with identity which are principal ideal rings. In the more general case of non-commutative algebras, Jespers and Okni\'nski \cite{MR1401676} described all principal ideal semigroup algebras. Subsequently Ara\'{u}jo et al. \cite{MR2100354} developed an algorithm which, given a presentation for a commutative semigroup $S$ and the characteristic of a field $\Bbbk$, decides whether the semigroup algebra $\Bbbk [S]$ is a principal ideal ring with identity.

Catoiu et al. \cite{MR1614186,MR1856919,MR3646318} showed that some (quantized) enveloping algebras are principal ideal rings. In particular, Siciliano and Usefi \cite{MR3646318} raised an open question that under what condition a Hopf algebra is a principal ideal ring. We settle the problem for an significant class of Hopf algebras. Let $H$ be a finite-dimensional Hopf algebra of rank one over an algebraically closed field of characteristic zero. We \cite{MR4256338} proved that $H$ is a principle ideal ring. As is well known, the Radford Hopf algebras provide examples of finite-dimensional pointed Hopf algebras of rank one. Then in \cite{MR4296825} we described all ideals and annihilator ideals of indecomposable modules of the Radford Hopf algebras by generators. Moreover, we \cite{yuwang,MR4296825} classified all ideals of $8$-dimensional and $9$-dimensional Radford Hopf algebras. Based on the results of \cite{MR4256338}, this paper is devoted to characterizing annihilator ideals of indecomposable $H$-modules by virtue of the primitive central orthogonal idempotents of associate group algebra.

This paper is organized as follows. In Section $2$, we recall from \cite{MR3284336,MR3448167} all indecomposable modules of a finite-dimensional pointed Hopf algebra of rank one. In Section $3$, we describe the generators of all ideals of $H$ by using the central idempotents of associate group algebra. In Section $4$, we give the classification of annihilator ideals of indecomposable $H$-modules. When $H$ is of nilpotent type, we also determine the completely prime ideals. In Section $5$, we give explicitly the classification of all ideals of finite-dimensional pointed Hopf algebra of rank one of nilpotent type over Klein $4$-group.

Throughout, we work over an algebraically closed field $\Bbbk$ of characteristic zero. Unless other stated, all algebras, Hopf algebras and modules are vector spaces over $\Bbbk$; all modules are finite-dimensional left modules; all maps are $\Bbbk$-linear;  $\otimes$ means $\otimes_{\Bbbk}$. For a group $G$, denote by $Z(G)$ the center of $G$ and $Z(\Bbbk G)$ the center of the group algebra $\Bbbk G$. We assume that the reader is familiar with the basics of Hopf algebras and representation theory. References \cite{MR1321145,MR1243637} are suggested for the former and \cite{MR2197389,MR1314422} are suggested for the latter.

\section{Indecomposable modules}
\indent

In this section, we retrospect the construction of a finite-dimensional pointed Hopf algebra of rank one from a group datum $\mathcal{D}$. A quadruple $\mathcal{D}=(G,\chi,g,\alpha)$ is called a group datum if $G$ is a finite group, $\chi$ a $\Bbbk$-linear character of $G$, $g\in Z(G)$, and $\alpha\in\Bbbk$, provided $\alpha(g^n-1)=0$ or $\chi^n=1$, where $n$ is the order of $\chi(g)$. If $\alpha(g^n-1)=0$, we say that the group datum $\mathcal{D}$ is of nilpotent type; otherwise $\mathcal{D}$ is of non-nilpotent type (see \cite{MR2236601}). Since $\Bbbk$ is algebraically closed, we may assume that $\alpha=0$ or $\alpha=1$.

Given a group datum $\mathcal{D}=(G,\chi,g,\alpha)$, we let $H_{\mathcal{D}}$ be an associative algebra generated by $z$ and all $s$ in $G$ such that $\Bbbk G$ is a subalgebra of $H_{\mathcal{D}}$ and
\begin{equation}\label{relation}
z^n=\alpha(g^n-1), \quad zs=\chi(s)sz \quad \text{for}\ s\in G.
\end{equation}
The algebra $H_{\mathcal{D}}$ is finite-dimensional with a $\Bbbk$-basis $\{z^lh\mid 0\leqslant l\leqslant n-1,\, h\in G\}$. Therefore $\dim H_{\mathcal{D}}=n|G|$, where $|G|$ is the order of $G$.

We can endow $H_{\mathcal{D}}$ with a Hopf algebra structure, where the comultiplication $\Delta$, the counit $\varepsilon$, and the antipode $S$ are given respectively by
\begin{align}
\Delta(z)&=z\otimes g+1\otimes z,\quad & \varepsilon(z)&=0,\quad & S(z)&=-zg^{-1},\label{comu}\\
\Delta(s)&=s\otimes s,\quad & \varepsilon(s)&=1,\quad & S(s)&=s^{-1}\label{antipode}
\end{align}
for all $s\in G$.
It is easy to see that $H_{\mathcal{D}}$ is a pointed Hopf algebra of rank one  with $G$ being the group of group-like elements of $H_{\mathcal{D}}$ (see \cite{MR2236601}). Moreover, $H_{\mathcal{D}}$ is called of nilpotent type if $\mathcal{D}$ is so, and of non-nilpotent type, otherwise.
If $n=1$, then the Hopf algebra $H_{\mathcal{D}}$ equals the group algebra $\Bbbk G$. To avoid this circumstances we assume that $n\geqslant2$ throughout this paper, which implies that $g\neq1$ and $\chi\neq\varepsilon$.

For any $h\in \Bbbk G$ with $\Delta(h)=\sum h_1\otimes h_2,$
the $\Bbbk$-linear character $\chi$ induces an automorphism $\sigma$ of $\Bbbk G$ as $\sigma(h)=\sum\chi(h_1)h_2.$
In this case,
\begin{equation}\label{eequ2.4}
z^mh=\sigma^m(h)z^m\ \text{for}\ 0\leqslant m\leqslant n-1.
\end{equation}

The proposition below provides a classification of Hopf algebras $H_{\mathcal{D}}$ with group data $\mathcal{D}$ (see \cite[Theorem 1]{MR2236601} or \cite[Theorem 5.9]{MR2047446}).
\begin{proposition}\cite[Theorem 1]{MR2236601}\label{2.1}
Let $\mathcal{D}=(G,\chi,g,\alpha)$ and $\mathcal{D}^\prime=(G^\prime,\chi^\prime,g^\prime,\alpha^\prime)$ be two group data. Then the Hopf algebras $H_{\mathcal{D}}$ and $H_{\mathcal{D}^\prime}$ are isomorphic as Hopf algebras if and only if there is a group isomorphism $f:G\rightarrow G^\prime$ such that $f(g)=g^\prime$, $\chi=\chi^\prime\circ f$ and $\beta\alpha^\prime(g^{\prime n}-1)=\alpha(g^{\prime n}-1)$ for some non-zero $\beta\in\Bbbk$, where $n$ is the order of $\chi(g)$. And every finite-dimensional pointed Hopf algebra of rank one over $\Bbbk$ is isomorphic to $H_{\mathcal{D}}$ for some datum ${\mathcal{D}}$.
\end{proposition}
In the sequel, we recall the classification of indecomposable modules of finite-dimensional pointed Hopf algebras of rank one of nilpotent type and non-nilpotent type, respectively. 

{\bf Case $1$}: The group datum $\mathcal{D}=(G,\chi,g,\alpha)$ is of nilpotent type, namely, $\alpha(g^n-1)=0$, where $n$ is the order of $q:=\chi(g)$. In this case, it is either $\alpha=0$ or $g^n-1=0$. In both situations, Proposition \ref{2.1} implies that the Hopf algebras constructed from  $(G,\chi,g,\alpha)$ and $(G,\chi,g,0)$ respectively are isomorphic. Hence we may suppose that $\alpha=0$ for any group datum $\mathcal{D}=(G,\chi,g,0)$ of nilpotent type. For simplicity, we shall drop the subscript $\mathcal{D}$ from $H_{\mathcal{D}}$, and denote by $H$ the Hopf algebra $H_{\mathcal{D}}$ associated with the group datum $\mathcal{D}=(G,\chi,g,0)$.

Since the Jacobson radical of $H$ is $J=(z)$ and $H/J\cong\Bbbk G$, an $H$-module is simple if and only if the restriction of it to $\Bbbk G$ is simple. Hence a complete set of non-isomorphic simple $\Bbbk G$-modules forms a complete set of non-isomorphic simple $H$-modules. Assume that $G$ has $p$ non-equivalent irreducible characters and $\Omega_0=\{0,1,\cdots, p-1\}$. We fix such a complete set \{$V_i\mid i\in\Omega_0$\} of non-isomorphic simple $\Bbbk G$ (and $H$)-modules.

Let $V_\chi$ and $V_{\chi^{-1}}$ be two $1$-dimensional simple $\Bbbk G$-modules corresponding to the $\Bbbk$-linear character $\chi$ and $\chi^{-1}$ respectively. For any simple $\Bbbk G$-module $V_i$, the tensor product $V_{\chi^{-1}}\otimes V_i\cong V_i\otimes V_{\chi^{-1}}$ is simple as well. Therefore, there is a unique permutation $\tau$ of $\Omega_0$ such that
$$V_{\chi^{-1}}\otimes V_i\cong V_i\otimes V_{\chi^{-1}}\cong V_{\tau(i)}.$$
The inverse of $\tau$ is determined by
$$V_{\chi}\otimes V_i\cong V_i\otimes V_{\chi}\cong V_{\tau^{-1}(i)}.$$

Let $x$ be a variable and $V$ a $\Bbbk G$-module. For any non-negative integer $l$, let $x^lV$ be a vector space defined by $x^lu+x^lv=x^l(u+v)$ and $\lambda(x^lu)=x^l(\lambda u)$ for all $u,v\in V$ and $\lambda\in\Bbbk$. Then $x^lV$ has a $\Bbbk G$-module structure defined by
\begin{equation}\label{Eq}
s\cdot(x^lv)=\chi^{-l}(s)x^l(s\cdot v)
\end{equation}
for any $s\in G$ and $v\in V$. For any simple $\Bbbk G$-module $V_i$, it is clear that there are $\Bbbk G$-module isomorphisms $x^lV_i\cong V_{\chi^{-l}}\otimes V_i\cong V_{\tau^l(i)}$.

For any $i\in\Omega_0$ and $1\leqslant k\leqslant n$, we consider the following $\Bbbk G$-module
$$M(k,i):=V_i\oplus xV_i\oplus\cdots\oplus x^{k-1}V_i,$$
where each summand is a simple $\Bbbk G$-module defined by (\ref{Eq}). In order to make $M(k,i)$ become an $H$-module, we
define an action of $z$ on $M(k,i)$ as follows:
\begin{equation}\label{EQ}
z\cdot(x^lv)=\begin{cases}
x^{l+1}v, & 0\leqslant l\leqslant k-2,\\
0, & l=k-1
\end{cases}
\end{equation}
for any $v\in V_i$. Then $M(k,i)$ is an $H$-module and $M(1,i)=V_i$. We have the following proposition (see \cite[Theorem 2.5]{MR3284336}).
\begin{proposition} \cite[Theorem 2.5]{MR3284336}\label{2.2}
Let $H$ be a finite-dimensional pointed Hopf algebra of rank one of nilpotent type.
Then the set \{$M(k,i)\mid i\in\Omega_0,\,1\leqslant k\leqslant n$\} forms a complete set of finite-dimensional indecomposable $H$-modules up to isomorphism. And the set $\{M(1,i)\mid i\in \Omega_0\}$ forms a complete set of finite-dimensional simple $H$-modules up to isomorphism.
\end{proposition}

{\bf Case 2}: The group datum $\mathcal{D}=(G,\chi,g,\alpha)$ is of non-nilpotent type, namely, $\chi^n=1$ and $\alpha(g^n-1)\neq0$, where $n$ is the order of $q=\chi(g)$. Without loss of generality, we may assume that $\alpha=1$ (see \cite[Corollary 1]{MR2236601}). Denote by $H$ the Hopf algebra $H_{\mathcal{D}}$ associated with the group datum $\mathcal{D}=(G,\chi,g,1)$. In this case, the relations in (\ref{relation}) become
$$z^n=g^n-1, \quad zs=\chi(s)sz \qquad \text{for}\ s\in G.$$
Note that the order of $q=\chi(g)$ is $n$. This implies that the order of $g$ in the group $G$ is $nr$, where $r>1$. Assume that $G$ has $p$ non-equivalent irreducible characters and $\Lambda=\{0,1,\cdots, p-1\}$.
Let \{$V_i\mid i\in\Lambda$\} be a complete set of simple $\Bbbk G$-modules up to isomorphism. Since $g^n\in Z(G)$, the action of $g^n$ on each $V_i$ as multiplied by a non-zero element of $\Bbbk$, say $\lambda_i$. Let $\Lambda_0=\{i\in\Lambda\mid\lambda_i=1\}$ and $\Lambda_1=\{i\in\Lambda\mid\lambda_i\neq1\}$.

Let $M(k,i):=V_i\oplus xV_i\oplus\cdots\oplus x^{k-1}V_i$ for $i\in\Lambda_0$ and $1\leqslant k\leqslant n$. Then $M(k,i)$ is an $H$-module where the action of each $s\in G$ on $M(k,i)$ is defined by (\ref{Eq}) and the action of $z$ on $M(k,i)$ is defined by (\ref{EQ}).

Let $P_j=V_j\oplus xV_j\oplus\cdots\oplus x^{n-1}V_j$ for $j\in\Lambda_1$. Then $P_j$ is an $H$-module, where the action of each $s\in G$ on $P_j$ is given similar to (\ref{Eq}) and $z$ acts on $P_j$ as follows:
\begin{equation}
z\cdot(x^lv)=\begin{cases}
x^{l+1}v, & 0\leqslant l\leqslant n-2,\\
(\lambda_j-1)v, & l=n-1
\end{cases}
\end{equation}
for any $v\in V_j$.

Let $V_\chi$ and $V_{\chi^{-1}}$ be two $1$-dimensional simple $\Bbbk G$-modules corresponding to the $\Bbbk$-linear character $\chi$ and $\chi^{-1}$ respectively. Similar to the case that $H$ is of nilpotent type, there is a unique permutation $\tau$ of the index $\Lambda$ determined by
$$V_{\chi^{-1}}\otimes V_t\cong V_t\otimes V_{\chi^{-1}}\cong V_{\tau(t)}$$
as $\Bbbk G$-modules for some $\tau(t)\in\Lambda$. Moreover, $\tau$ preserves $\Lambda_0$ and $\Lambda_1$ respectively. We have the following lemma (see \cite[Lemma 2.7]{MR3448167}).
\begin{lemma}\cite[Lemma 2.7]{MR3448167}\label{2.0}
For any $t\in \Lambda$ and $l\in \mathbb{Z}$, the following hold for the $\Bbbk G$-modules:
\begin{enumerate}
\item[(1)] $V_t\otimes V_{\chi}\cong V_{\chi}\otimes V_t\cong V_{\tau^{-1}(t)}.$
\item[(2)] $V_t\otimes V_{\chi^{-l}}\cong V_{\tau^{l}(t)}$.
\item[(3)] $xV_t\cong V_{\tau(t)}$. Moreover, $V_i\cong V_j$ if and only if $xV_i\cong xV_j$ for $i,j\in \Lambda$.
\item[(4)] The order of the permutation $\tau$ is $n$. Moreover, for any $t\in \Lambda$, $x^lV_t\cong V_t$ if and only if $l$ is divisible by $n$.
\end{enumerate}
\end{lemma}
Let $\langle \tau\rangle$ be the cyclic group generated by $\tau$. Then $\langle \tau\rangle$ acts on the sets $\Lambda_k$ for $k\in\{0,1\}$, respectively. Let $\sim$ be the equivalence relation on $\Lambda_k$: $i\sim j$ if and only if $i=\tau^m(j)$ for some $m\in\mathbb{N}$. Denote by $[i]$ the equivalence class of $i$. By Lemma \ref{2.0}, we have $|[i]|=n$ for any $i\in \Lambda_k$ and $|\Lambda_k|$ is divisible by $n$.

We have the following proposition (see \cite[Proposition 2.4, Proposition 2.8 and Theorem 2.9]{MR3448167}).
\begin{proposition}\cite[Proposition 2.4, Proposition 2.8 and Theorem 2.9]{MR3448167} Let $H$ be a finite-dimensional pointed Hopf algebra of rank one of non-nilpotent type. Then with notations as above, for any two $j,j^\prime\in\Lambda_1$, $P_j\cong P_{j^\prime}$ as $H$-modules if and only if $[j]=[j^\prime]$. The set $\{M(k,i),P_{[j]}\mid i\in \Lambda_0,\,1\leqslant k\leqslant n,\,j\in \Lambda_1\}$ forms a complete set of finite-dimensional indecomposable $H$-modules up to isomorphism. And the set $\{M(1,i),P_{[j]}\mid i\in \Lambda_0,\,j\in \Lambda_1\}$ forms a complete set of finite-dimensional simple $H$-modules up to isomorphism.
\end{proposition}

\section{Generators of all ideals}

\indent

In this section, $H$ is always a finite-dimensional pointed Hopf algebra of rank one. We denote $(a)$ the two-sided ideal of $H$ generated by $a$ for any $a\in H$ and  $\langle h\rangle$ the two-side ideal of $\Bbbk G$ for any $h\in \Bbbk G$. For the generators of any non-zero two-sided ideal of $H$, we have the following lemma (see \cite[Proposition 3.2]{MR4256338}).

\begin{lemma}\cite[Proposition 3.2]{MR4256338}\label{3.1}
Let $I$ be any non-zero two-sided ideal of $H$. Then there exist integers $t>0$, $0\leqslant m_1,m_2,\cdots,m_t\leqslant n-1$ and $0\neq h_i\in\Bbbk G$ for $1\leqslant i\leqslant t$, such that
$$I=(z^{m_1}h_1,z^{m_2}h_2,\cdots,z^{m_t}h_t).$$
\end{lemma}

\begin{lemma}\label{3.2}
For $h_1,h_2\in \Bbbk G$ and $0\leqslant m\leqslant n-1$, we have that $(z^mh_1,z^mh_2)=(z^mh)$ for some $h\in Z(\Bbbk G)$ and $h^2=h$.
\end{lemma}
\begin{proof}
Since $\Bbbk G$ is semisimple, it follows from \cite[Exercise 3.2.3]{MR0674652} that any two-sided ideal of $\Bbbk G$ is generated by a central idempotent element. Namely, $\langle h_1, h_2\rangle=\langle h\rangle$, where $h\in Z(\Bbbk G)$ and $h^2=h$. Hence there exists an $a\in \Bbbk G$ such that $h_1=ah=ha$. Therefore we have $z^mh_1=z^mha\in (z^mh).$ The proof of $z^mh_2\in (z^mh)$ is similar. Hence $(z^mh_1,z^mh_2)\subseteq(z^mh).$
For the converse inclusion, we suppose that
$$h=\sum_i(c_ih_1d_i+e_ih_2f_i)$$
for $c_i,d_i,e_i,f_i\in \Bbbk G$. Then $$z^mh=\sum_iz^m(c_ih_1d_i+e_ih_2f_i)
=\sum_i(\sigma^m(c_i)z^mh_1d_i+\sigma^m(e_i)z^mh_2f_i)\in(z^mh_1,z^mh_2).$$ This implies that $(z^mh)\subseteq(z^mh_1,z^mh_2)$. Hence $(z^mh_1,z^mh_2)=(z^mh)$.
\end{proof}

\begin{proposition}\label{3.3}
Let $I$ be any non-zero two-sided ideal of $H$. There exist integers $1\leqslant t\leqslant n, n-1\geqslant m_1>m_2>\cdots>m_t\geqslant0$, $d_i\in Z(\Bbbk G), d_i^2=d_i$ for $1\leqslant i\leqslant t$, and $\langle d_1\rangle\supsetneq \langle d_2\rangle\supsetneq\cdots\supsetneq \langle d_t\rangle\supsetneq \langle 0\rangle$ such that
$$I=(z^{m_1}d_1,z^{m_2}d_2,\cdots,z^{m_t}d_t)=(z^{m_1}d_1+z^{m_2}d_2+\cdots+z^{m_t}d_t).$$
\end{proposition}
\begin{proof} For the first equality, by Lemma \ref{3.1} we have $I=(z^{m_1}h_1,z^{m_2}h_2,\cdots,z^{m_t}h_t)$, where $0\neq h_i\in \Bbbk G$ for $1\leqslant i\leqslant t.$ Without loss of generality, we may assume that $n-1\geqslant m_1\geqslant m_2\geqslant\cdots\geqslant m_t\geqslant 0$. If $m_i=m_{i+1}$ for some $1\leqslant i\leqslant t-1$, by Lemma \ref{3.2} we have that $(z^{m_i}h_i,z^{m_{i+1}}h_{i+1})=(z^{m_i}h)$ for some $h\in Z(\Bbbk G)$ and $h^2=h$. Therefore we can replace the generators $z^{m_i}h_i$ and $z^{m_{i+1}}h_{i+1}$ of $I$ by $z^{m_i}h$. Since the number of generators of $I$ is finite, we can
continue the above steps finite times until these numbers $m_i$ are all different. Since
$$(z^{m_i}h_i)=(z^{m_i}h_i, z^{m_i}0)=(z^{m_i}h^{\prime}),$$
where $h^{\prime}\in Z(\Bbbk G)$ and $h^{\prime 2}=h^{\prime}\neq 0$, we may suppose that $h_i\in Z(\Bbbk G)$ and $h_i^2=h_i\neq 0$ for $1\leqslant i\leqslant t$. Denote by $d_t:=h_t$. Noting that $m_{t-1}>m_t$, we have
$$z^{m_{t-1}}d_t\in (z^{m_t}d_t)\subseteq (z^{m_{t-1}}h_{t-1}, z^{m_t}d_t).$$
Hence
\begin{align*}
I&=(z^{m_1}h_1,\cdots,z^{m_{t-1}}h_{t-1},z^{m_t}d_t)\\
&=(z^{m_1}h_1,\cdots,z^{m_{t-2}}h_{t-2},z^{m_{t-1}}h_{t-1},z^{m_{t-1}}d_t,z^{m_t}d_t)\\
&=(z^{m_1}h_1,\cdots,z^{m_{t-2}}h_{t-2})+(z^{m_{t-1}}h_{t-1},z^{m_{t-1}}d_t)+(z^{m_t}d_t)\\
&=(z^{m_1}h_1,\cdots,z^{m_{t-2}}h_{t-2})+(z^{m_{t-1}}d_{t-1})+(z^{m_t}d_t)\\
&=(z^{m_1}h_1,\cdots,z^{m_{t-2}}h_{t-2},z^{m_{t-1}}d_{t-1},z^{m_t}d_t),
\end{align*}
where $d_{t-1}\in Z(\Bbbk G)$, $d_{t-1}^2=d_{t-1}$ and $\langle d_{t-1}\rangle=\langle h_{t-1},d_t\rangle\supseteq \langle d_t \rangle$.
We continue the above steps until we get
$$I=(z^{m_1}d_1,z^{m_2}d_2,\cdots z^{m_t}d_t),$$
where $\langle d_1\rangle\supseteq \langle d_2\rangle\supseteq \cdots \supseteq \langle d_t\rangle, d_i\in Z(\Bbbk G)$ and $d_i^2=d_i$ for $1\leqslant i\leqslant t$.
If $\langle d_i\rangle=\langle d_{i+1}\rangle$ for some $1\leqslant i\leqslant t-1$, then $d_i=cd_{i+1}=d_{i+1}c$ for some $c\in \Bbbk G$. Thus,
$$(z^{m_i}d_i,z^{m_{i+1}}d_{i+1})=(z^{m_{i}}d_{i+1}c, z^{m_{i+1}}d_{i+1})=(z^{m_{i+1}}d_{i+1}).$$
We can remove the redundant generator $z^{m_i}d_i$. Since $I$ is finitely generated, the above steps could only continue finite times. Therefore we may assume that
$$\langle d_1\rangle\supsetneq \langle d_2\rangle\supsetneq\cdots\supsetneq \langle d_t\rangle\supsetneq \langle 0\rangle.$$
We have proven the first equality. The proof of the second equality is similar to that of \cite[Theorem 3.4]{MR4256338}.
\end{proof}

\begin{lemma}\label{3.4}
Let $H$ be of nilpotent type. Then the following ideals are all different:
$$(h_0),\quad (zh_1),\quad \cdots,\quad (z^{n-1}h_{n-1}),$$
where $0\neq h_i\in \Bbbk G$ for $0\leqslant i\leqslant n-1$.
\end{lemma}

\begin{proof}
Note that for $0\leqslant i\leqslant n-1$ each element of $(z^ih_i)$ is of the form $z^iu$, $u\in H$. If $0\leqslant k<l\leqslant n-1$, then $z^kh_k\notin (z^lh_l)$. Hence $(z^kh_k)\neq (z^lh_l)$.
\end{proof}

\begin{proposition}\label{3.5}
Let $H$ be of nilpotent type. If
$$(z^{l_1}h_1+\cdots+z^{l_t}h_t)=(z^{m_1}a_1+\cdots+z^{m_r}a_r),$$
where $n-1\geqslant l_1>\ldots>l_t\geqslant 0$, $n-1\geqslant m_1>\cdots>m_r\geqslant 0$, $h_i, a_j\in \Bbbk G$, $1\leqslant i\leqslant t$, $1\leqslant j\leqslant r$, $h_t\neq0,$ $a_r\neq0$, then $l_t=m_r$ and $\langle h_t \rangle=\langle a_r \rangle$.
\end{proposition}

\begin{proof}
Note that each element of $(z^{l_1}h_1+\cdots+z^{l_t}h_t)$ is of the form $z^{l_t}v$, $v\in H$. If $l_t>m_r$, then $z^{m_r}a_r\notin (z^{l_1}h_1+\cdots+z^{l_t}h_t)$, which is a contradiction. Therefore, $l_t\leqslant m_r$. Similarly we can prove $m_r\leqslant l_t$. Hence $l_t=m_r$. We may suppose
$$z^{l_1}h_1+\cdots+z^{l_t}h_t=\sum_s(z^{n-1}p_{1s}+\cdots+p_{ns})(z^{m_1}a_1+\cdots+z^{m_r}a_r)(z^{n-1}q_{1s}+\cdots+q_{ns}),$$
where $p_{is}, q_{is}\in \Bbbk G$, $1\leqslant i\leqslant n$. Since $\{z^ih\mid 0\leqslant i\leqslant n-1,\, h\in G\}$ is a $\Bbbk$-basis of $H$, it follows
$$h_t=\sum_s \sigma^{-m_r}(p_{ns})a_rq_{ns}.$$
Thus, $h_t\in \langle a_r\rangle$. That is to say, $\langle h_t\rangle\subseteq \langle a_r\rangle$. Similarly, we can prove $\langle a_r\rangle\subseteq \langle h_t\rangle$. Hence $\langle h_t\rangle=\langle a_r\rangle$.
\end{proof}

In the following, we use the primitive central orthogonal idempotents of $\Bbbk G$ to describe the ideals of $H$. Let $\{\chi_i\mid 0\leqslant i\leqslant p-1\}$ be a complete set of irreducible characters of $G$ and
\begin{equation*}
e_i=\frac{\chi_i(1)}{|G|}\sum_{h\in G}\chi_i(h)h^{-1} \quad \text{for}\,\, 0\leqslant i\leqslant p-1.
\end{equation*}
Then $\{e_i\mid 0\leqslant i\leqslant p-1\}$ is a complete set of primitive central orthogonal idempotents of $\Bbbk G$ and a $\Bbbk$-basis of $Z(\Bbbk G)$. If $h\in Z(\Bbbk G)$ and $h^2=h$, then we may assume that $h=\alpha_0e_0+\alpha_1e_1+\cdots+\alpha_{p-1}e_{p-1}$, $\alpha_i\in \{0,1\}$ for $0\leqslant i\leqslant p-1$. We have the following corollary.

\begin{corollary}\label{3.7}
Let $H$ be of nilpotent type. If
\begin{align*}
&(z^{l_1}(\alpha_0^{(1)}e_0+\cdots+\alpha_{p-1}^{(1)}e_{p-1})+\cdots+z^{l_t}(\alpha_0^{(t)}e_0+\cdots+\alpha_{p-1}^{(t)}e_{p-1}))\\
=&(z^{m_1}(\beta_0^{(1)}e_0+\cdots+\beta_{p-1}^{(1)}e_{p-1})+\cdots+z^{m_r}(\beta_0^{(r)}e_0+\cdots+\beta_{p-1}^{(r)}e_{p-1})),
\end{align*}
where $n-1\geqslant l_1>\ldots>l_t\geqslant 0$, $n-1\geqslant m_1>\cdots>m_r\geqslant 0$, $\alpha_i^{(j)}, \beta_i^{(s)}\in \{0,1\}$, $0\leqslant i\leqslant p-1$, $1\leqslant j\leqslant t$, $1\leqslant s\leqslant r$, $\alpha_0^{(t)}+\cdots+\alpha_{p-1}^{(t)}\neq0$, $\beta_0^{(r)}+\cdots+\beta_{p-1}^{(r)}\neq0$, then $l_t=m_r$ and $\alpha_i^{(t)}=\beta_i^{(r)}$ for $0\leqslant i\leqslant p-1$.
\end{corollary}
\begin{proof}
It is clear by Proposition \ref{3.5}.
\end{proof}

\begin{corollary}\label{3.8}
Let $H$ be of nilpotent type. Then for any $0\leqslant m\leqslant n-1$, the following $2^p-1$ ideals are all different
\begin{gather*}
(z^m),\quad (z^me_0),\quad \cdots, \quad (z^me_{p-1}),\\
(z^m(e_0+e_1)), \quad (z^m(e_0+e_2)), \cdots,\quad (z^m(e_{p-2}+e_{p-1})),\\
\cdots,\quad \cdots\\
(z^m(e_0+e_1+\cdots+e_{p-2})), \quad \cdots,\quad (z^m(e_1+e_2+\cdots+e_{p-1})).
\end{gather*}
\end{corollary}
\begin{proof}
By Corollary \ref{3.7}, it follows that if
$$(z^m(l_0e_0+\cdots+l_{p-1}e_{p-1}))=(z^m(r_0e_0+\cdots+r_{p-1}e_{p-1})),$$
where $l_i, r_i\in \{0, 1\}$ for $0\leqslant i\leqslant p-1$, $l_0+\cdots+l_{p-1}\neq0$, $r_0+\cdots+r_{p-1}\neq0$, then $l_i=r_i$ for $0\leqslant i\leqslant p-1$. Hence we finish the proof.
\end{proof}

\section{Annihilator ideals of indecomposable modules}
\indent

In this section, we use the primitive central orthogonal idempotents of $\Bbbk G$ to describe the annihilator ideals of indecomposable $H$-modules.
Since $V_s$ is a simple $\Bbbk G$-module, we suppose that $V_s=\Bbbk G\cdot v_s$ for any non-zero $v_s\in V_s$. Through direct calculations we have $e_m\cdot v_s=\delta_{m,s}v_s$ for $0\leqslant m,s\leqslant p-1$, where $\delta_{m,s}$ is Kronecker symbol. Since $g\in Z(G)$, by Schur's lemma, $g$ acts on $V_s$ as multiplied by a non-zero scalar $\gamma_s$. We have the following lemma.

\begin{lemma}\label{4.1}
$z^{l+1}\in (z^l(1-e_s))$ for $0\leqslant s\leqslant p-1$ and $0\leqslant l\leqslant n-1$.
\end{lemma}
\begin{proof}
Note that $g-\gamma_s\in Z(\Bbbk G)$ and $\{e_s\mid 0\leqslant s\leqslant p-1\}$ is a $\Bbbk$-basis of $Z(\Bbbk G)$. We may assume that
$$g-\gamma_s=k_0e_0+k_1e_1+\cdots+k_{p-1}e_{p-1},$$
where $k_i\in \Bbbk, 0\leqslant i\leqslant p-1$. Noticing that $(g-\gamma_s)\cdot v_s=0$, we have
$$(k_0e_0+\cdots+k_{p-1}e_{p-1})\cdot v_s=k_sv_s=0.$$
It follows $k_s=0$. Hence we have
$$g-\gamma_s\in (e_0,\cdots, e_{s-1}, e_{s+1}, \cdots, e_{p-1})=(e_0+\cdots+e_{s-1}+e_{s+1}+\cdots+e_{p-1})=(1-e_s).$$
Since $z(g-\gamma_s)-(g-\gamma_s)z=(q-1)gz\in(g-\gamma_s)$ and $g$ is invertible, it follows $z\in (g-\gamma_s)\subseteq (1-e_s)$. Hence we suppose
$$z=\sum_a(z^{n-1}h_{1a}+\cdots+h_{na})(1-e_s)(z^{n-1}h_{1a}^{\prime}+\cdots+h_{na}^{\prime}),$$
where $h_{ia},h_{ia}^{\prime}\in \Bbbk G$, $1\leqslant i\leqslant n$. Thus,
\begin{align*}
z^{l+1}&=z^l\sum_a(z^{n-1}h_{1a}+\cdots+h_{na})(1-e_s)(z^{n-1}h_{1a}^{\prime}+\cdots+h_{na}^{\prime})\\
&=\sum_a(z^{n-1}\sigma^l(h_{1a})+\cdots+\sigma^l(h_{na}))z^l(1-e_s)(z^{n-1}h_{1a}^{\prime}+\cdots+h_{na}^{\prime}).
\end{align*}
Therefore we finish the proof.
\end{proof}

\begin{lemma}\label{101}
For $0\leqslant s\leqslant p-1$ and $0\leqslant l\leqslant n-1$, we have $z^le_s=e_{\tau^l(s)}z^l$. When $H$ is of nilpotent type (or non-nilpotent type), we have
\begin{equation*}
e_s\cdot(x^lv_i)=\delta_{s, \tau^l(i)}x^lv_i \quad \text{for}\,\, i\in \Omega_0\  (\text{or}\ i\in\Lambda_0).
\end{equation*}
When $H$ is of non-nilpotent type, we have
\begin{equation*}
e_{s}\cdot (x^lv_j)=\delta_{s, \tau^l(j)}x^lv_j \quad \text{for}\,\, j\in \Lambda_1.
\end{equation*}
\end{lemma}
\begin{proof}
Note that $\chi^{-l}\chi_s=\chi_s\chi^{-l}=\chi_{\tau^l(s)}$ for $0\leqslant l\leqslant n-1$. We have that the degree of $\chi_s$ is equal to the degree of $\chi_{\tau^l(s)}$, namely, $\chi_s(1)=\chi_{\tau^l(s)}(1)$. Hence we obtain that
\begin{align*}
z^le_s&=\frac{\chi_s(1)}{|G|}\sum_{h\in G}\chi_s(h)z^lh^{-1}=\frac{\chi_s(1)}{|G|}\sum_{h\in G}\chi_s(h)\chi^{-l}(h)h^{-1}z^l \\
&=\frac{\chi_{\tau^l(s)}(1)}{|G|}\sum_{h\in G}\chi_{\tau^l(s)}(h)h^{-1}z^l=e_{\tau^l(s)}z^l.
\end{align*}
When $H$ is of nilpotent type (or non-nilpotent type), for $i\in \Omega_0$ (or $i\in \Lambda_0$) and $v_i\in V_i$, it follows that
\begin{align*}
e_s\cdot(x^lv_i)&=\frac{\chi_s(1)}{|G|}\sum_{h\in G}\chi_s(h)h^{-1}\cdot (x^lv_i)=x^l(\frac{\chi_s(1)}{|G|}\sum_{h\in G}\chi_s(h)\chi^{-l}(h^{-1})h^{-1}\cdot v_i)\\
&=x^l(\frac{\chi_s(1)}{|G|}\sum_{h\in G}\chi_s(h)\chi^l(h)h^{-1}\cdot v_i)=x^l(\frac{\chi_{\tau^{-l}(s)}(1)}{|G|}\sum_{h\in G}\chi_{\tau^{-l}(s)}(h)h^{-1}\cdot v_i)\\
&=x^l(e_{\tau^{-l}(s)}\cdot v_i).
\end{align*}
Hence we have $e_{s}\cdot (x^lv_i)=\delta_{s,\tau^l(i)} x^lv_i$. Similarly, when $H$ is of non-nilpotent type, $e_{s}\cdot (x^lv_j)=\delta_{s, \tau^l(j)}x^lv_j$ for $j\in \Lambda_1$ and $0\leqslant l\leqslant n-1$.
\end{proof}

\begin{theorem}\label{4.2}
Let $H$ be of nilpotent type (or non-nilpotent type). Then for $1\leqslant k\leqslant n$ and $i\in \Omega_0$ (or $i\in \Lambda_0$), the annihilator ideal of $M(k,i)$ is
$$(z^{k-1}(1-e_i),\,z^{k-2}(1-e_i)(1-e_{\tau(i)}),\,\cdots,\,(1-e_i)(1-e_{\tau(i)})\cdots(1-e_{\tau^{k-1}(i)})).$$
\end{theorem}
\begin{proof}
Let $I=(z^{m_1}h_1+z^{m_2}h_2\cdots+z^{m_t}h_t)$ be the annihilator ideal of $M(k,i)$, where $n-1\geqslant m_1>\cdots>m_t\geqslant0$, $h_s\in Z(\Bbbk G)$, $h_s^2=h_s$ for $1\leqslant s\leqslant t$. Note that $M(k,i)=V_i\oplus xV_i\oplus\cdots\oplus x^{k-1}V_i$ for $i\in \Omega_0$ (or $i\in \Lambda_0$), where $V_i=\Bbbk G\cdot v_i$ for any non-zero $v_i\in V_i$. According to Lemma \ref{101}, we have that
\begin{gather*}
(1-e_i)\cdot v_i=0,\quad (1-e_{\tau(i)})\cdot (xv_i)=0,\quad \cdots,\quad (1-e_{\tau^{k-1}(i)})\cdot (x^{k-1}v_i)=0.
\end{gather*}
When $0\leqslant c\leqslant k-1$ and $0\leqslant l\leqslant k-1-c$, it is clear that
$$(1-e_i)\cdots(1-e_{\tau^{k-1-c}(i)})\cdot(x^lv_i)=0.$$
Hence $z^{c}(1-e_i)\cdots(1-e_{\tau^{k-1-c}(i)})\cdot(x^lv_i)=0$. When $1\leqslant c\leqslant k-1$ and $k-c\leqslant l\leqslant k-1$, noticing that $(1-e_i)\cdots(1-e_{\tau^{k-1-c}(i)})$ acts on $x^lv_i$ as multiplied by a scalar and $z^{c}\cdot (x^lv_i)=0$, we have $z^{c}(1-e_i)\cdots(1-e_{\tau^{k-1-c}(i)})\cdot(x^lv_i)=0$. Therefore, for any $0\leqslant l\leqslant k-1$, it follows
$$z^{c}(1-e_i)\cdots(1-e_{\tau^{k-1-c}(i)})\cdot(x^lv_i)=0.$$
Hence $(z^{k-1}(1-e_i),\,z^{k-2}(1-e_i)(1-e_{\tau(i)}),\,\cdots,\,(1-e_i)\cdots(1-e_{\tau^{k-1}(i)}))\subseteq I.$ Since $I$ is the annihilator ideal of $M(k,i)$, it follows that
\begin{gather*}
z^{m_s}h_s\cdot v_i=0, \quad z^{m_s}h_s\cdot (xv_i)=0,\quad \cdots, \quad z^{m_s}h_s\cdot (x^{k-1}v_i)=0
\end{gather*}
for $1\leqslant s\leqslant t$. When $k\leqslant m_s\leqslant n-1$, it is clear $z^{m_s}h_s\cdot(x^lv_i)=0$ for any $h_s\in Z(\Bbbk G)$ and $0\leqslant l\leqslant k-1$. When $0\leqslant m_s\leqslant k-1$, noticing that $h_s$ acts on $x^lv_i$ as multiplied by a scalar and $z^{m_s}\cdot(x^lv_i)\neq0$ for $0\leqslant l\leqslant k-1-m_s$, we conclude $h_s\cdot (x^lv_i)=0$ for $0\leqslant l\leqslant k-1-m_s$. Let $h_s=\sum_{0\leqslant s\leqslant p-1}a_se_s$, $a_s\in \{0,1\}$. Since $h_s\cdot (x^lv_i)=a_{\tau^l(i)}=0$, we have 
$$z^{m_s}h_s\in\bigg(z^{m_s}\sum_{s\notin \{i,\tau(i),\cdots,\tau^{k-1-m_s}(i)\}}e_s\bigg)=(z^{m_s}(1-e_i)(1-e_{\tau(i)})\cdots(1-e_{\tau^{k-1-m_s}(i)})).$$
Hence
$$I\subseteq (z^k,\, z^{k-1}(1-e_i),\,\cdots,\, (1-e_i)(1-e_{\tau(i)})\,\cdots\,(1-e_{\tau^{k-1}(i)})).$$
By Lemma \ref{4.1}, it follows
$$I\subseteq (z^{k-1}(1-e_i), \,z^{k-2}(1-e_i)(1-e_{\tau(i)}),\,\cdots,\,(1-e_i)(1-e_{\tau(i)})\cdots(1-e_{\tau^{k-1}(i)})).$$
We complete the proof.
\end{proof}
In the sequel we discuss the annihilator ideal of simple $H$-module $P_j$, where $H$ is of non-nilpotent type and $j\in \Lambda_1$. Let $\mathbf{Orb}(j)=\{j,\tau(j),\cdots,\tau^{n-1}(j)\}$ be the orbit generated by $j$. We have the following theorem.
\begin{theorem}\label{4.3}
Let $H$ be of non-nilpotent type. Then for $j\in \Lambda_1$, the annihilator ideal of $P_j$ is
\begin{equation*}
\bigg(\sum_{i\in \Lambda_0}e_i+\sum_{j^\prime\in \Lambda_1-\mathbf{Orb}(j)}e_{j^\prime}\bigg).
\end{equation*}
\end{theorem}
\begin{proof}
Let $I=(z^{m_1}h_1+z^{m_2}h_2+\cdots+z^{m_t}h_t)$ be the annihilator ideal of $P_j$, where $n-1\geqslant m_1>m_2>\cdots>m_t\geqslant0$, $h_s\in Z(\Bbbk G)$, $h_s^2=h_s$ for $1\leqslant s\leqslant t$.
Note that $\tau$ preserve $\Lambda_0$ and $\Lambda_1$ respectively and $P_j=V_j\oplus xV_j\oplus\cdots\oplus x^{n-1}V_j$, where $V_j=\Bbbk G\cdot v_j$ for any non-zero $v_j\in V_j$. For any $i\in \Lambda_0$, $j^\prime\in \Lambda_1-\mathbf{Orb}(j)$ and $0\leqslant l\leqslant n-1$, it follows that $\tau^l(i)\in \Lambda_0$ and $\tau^l(j^\prime)\in \Lambda_1-\mathbf{Orb}(j)$. Therefore, $e_{\tau^l(i)}\cdot v_j=0$ and $e_{\tau^l(j^\prime)}\cdot v_j=0$. For $0\leqslant l\leqslant n-1$, we have that
\begin{align*}
&e_i\cdot (x^lv_j)=x^l(e_{\tau^{-l}(i)}\cdot v_j)=x^l(e_{\tau^{n-l}(i)}\cdot v_j)=0,\\
&e_{j^\prime}\cdot (x^lv_j)=x^l(e_{\tau^{-l}(j^\prime)}\cdot v_j)=x^l(e_{\tau^{n-l}(j^\prime)}\cdot v_j)=0.
\end{align*}
Since $P_j=V_j\oplus xV_j\oplus\cdots\oplus x^{n-1}V_j$, it is clear that $e_i\cdot P_j=0$ and $e_{j^\prime}\cdot P_j=0$. Hence $(\sum_{i\in \Lambda_0}e_i+\sum_{j^\prime\in \Lambda_1-\mathbf{Orb}(j)}e_{j^\prime})\cdot P_j=0$. Therefore it follows that
$$\bigg(\sum_{i\in \Lambda_0}e_i+\sum_{j^\prime\in \Lambda_1-\mathbf{Orb}(j)}e_{j^\prime}\bigg)\subseteq I.$$
Since $I$ is the annihilator ideal of $P_j$, it follows $z^{m_s}h_s\cdot (x^lv_j)=0$ for $1\leqslant s\leqslant t$ and $0\leqslant l\leqslant n-1$. Note that $h_s$ acts on $x^lv_j$ as multiplied by a scalar and $z^{m_s}\cdot(x^lv_j)\neq0$. Hence $h_s\cdot (x^lv_j)=0$ for $1\leqslant s\leqslant t$ and $0\leqslant l\leqslant n-1$. Noticing that $h_s\in Z(\Bbbk G)$ and $h_s^2=h_s$, we may suppose that
$$h_s=\sum_{i\in\Lambda_0}k_{is}e_i+\sum_{j^\prime\in \Lambda_1}k_{j^\prime s}e_{j^\prime},$$
where $k_{is}, k_{j^{\prime}s}\in \{0,1\}$. It is clear that $h_s\cdot (x^lv_j)=k_{\tau^l(j)s}=0$ for $0\leqslant l\leqslant n-1$. Hence
$$h_s=\sum_{i\in\Lambda_0}k_{is}e_i+\sum_{j^\prime\in \Lambda_1-\mathbf{Orb}(j)}k_{j^\prime s}e_{j^\prime}.$$
Thus, we obtain
$$z^{m_s}h_s=z^{m_s}\bigg(\sum_{i\in\Lambda_0}k_{is}e_i+\sum_{j^\prime\in \Lambda_1-\mathbf{Orb}(j)}k_{j^\prime s}e_{j^\prime}\bigg)\in \bigg(\sum_{i\in \Lambda_0}e_i+\sum_{j^\prime\in \Lambda_1-\mathbf{Orb}(j)}e_{j^\prime}\bigg).$$
It follows $I \subseteq (\sum_{i\in \Lambda_0}e_i+\sum_{j^\prime\in \Lambda_1-\mathbf{Orb}(j)}e_{j^\prime})$. We complete the proof.
\end{proof}

When $H$ is of nilpotent type (or non-nilpotent type), by Theorem \ref{4.2} the annihilator ideal of simple $H$-module $M(1,i)$ is $(1-e_i)$ for $i\in \Omega_0$ (or $i\in \Lambda_0$).
Hence by Theorems \ref{4.2} and \ref{4.3}, we deduce the following corollary.
\begin{corollary}\label{4.4}
When $H$ is of nilpotent type, the complete set of maximal ideals of $H$ is
$$\left\{(1-e_i)\mid i\in \Omega_0\right\}.$$
When $H$ is of non-nilpotent type, the complete set of maximal ideals of $H$ is
$$\bigg\{(1-e_i),\, \bigg(\sum_{i\in \Lambda_0}e_i+\sum_{j^\prime\in \Lambda_1-\mathbf{Orb}(j)}e_{j^\prime}\bigg) \,\,\bigg|\,\, i\in \Lambda_0,\, j\in \Lambda_1\bigg\}.$$
\end{corollary}
When $H$ is of nilpotent type, we also characterize all completely prime ideals of $H$.
\begin{theorem}
Let $H$ be of nilpotent type. Then the complete set of completely prime ideals of $H$ is $\{(1-e_i)\mid i\in \Omega_0\}.$
\end{theorem}
\begin{proof}
Let $I=(z^{m_1}h_1+z^{m_2}h_2+\cdots+z^{m_t}h_t)$ be any completely prime ideal, where $n-1\geqslant m_1>m_2>\cdots>m_t\geqslant0$, $h_s\in Z(\Bbbk G)$, $h_s^2=h_s$ for $1\leqslant s\leqslant t$. If $m_t>0$, then each element of $I$ is of the form $zv$, $v\in H$. It is easy to see that $e_0, e_1\notin I$ and $e_0e_1=0\in I$. It is in contradiction with the assumption that $I$ is completely prime. Therefore, $m_t=0$ and $I=(z^{m_1}h_1+z^{m_2}h_2+\cdots+h_t)$. We may assume that $h_t=\sum_{i\in B}e_i,$ where $B\subseteq \Omega_0=\{0,1,\cdots,p-1\}$. Hence
$$h_t=1-\sum_{i\in \Omega_0-B}e_i=\prod_{i\in \Omega_0-B}(1-e_i).$$
If $B=\Omega_0$, then $h_t=1$. In this case, we have $I=(1)$, which is in contradiction with $I$ being completely prime. Therefore, $B\subsetneq \Omega_0$. Note that $h_t=\prod_{i\in \Omega_0-B}(1-e_i)\in I$ and $I$ is completely prime. This would force $1-e_i\in I$ for some $i\in \Omega_0-B$. Since $(1-e_i)$ is maximal, it follows that $I=(1-e_i)$. Thus, we complete the proof.
\end{proof}

\section{When G is Klein $4$-group $K_4$}
\indent

In this section, we consider the case when $G=\langle b,c\mid b^2=c^2=1, bc=cb\rangle$ is Klein $4$-group $K_4$. It is clear that $\Bbbk K_4$ has four non-isomorphic $1$-dimensional simple modules $V_0=\Bbbk v_0$, $V_1=\Bbbk v_1$, $V_2=\Bbbk v_2$ and  $V_3=\Bbbk v_3$, where
\begin{align*}
\qquad&b\cdot v_0=v_0, \quad &c\cdot v_0=v_0; \quad\qquad &b\cdot v_1=v_1, \quad &c\cdot v_1=-v_1;\\
\qquad&b\cdot v_2=-v_2, \quad &c\cdot v_2=v_2; \quad\qquad &b\cdot v_3=-v_3, \quad &c\cdot v_3=-v_3.
\end{align*}
Let $\chi=\chi_3$ be the character of $V_3$ (i.e., $\chi(b)=\chi(c)=-1$). Then the order of $\chi$ is $2$. Consider the group datum $\mathcal{D}=(K_4,\chi,b,0)$ and let $H=H_{\mathcal{D}}$. Then $H$ is generated as an algebra by $b,c$ and $z$ subject to the following relations:
\begin{gather*}
zb=-bz, \qquad zc=-cz, \qquad z^2=0.
\end{gather*}
Let $M(k,i)=V_i\oplus\cdots\oplus x^{k-1}V_i$ for $0\leqslant i\leqslant 3$ and $1\leqslant k\leqslant2$. By Proposition \ref{2.2}, the complete set of non-isomorphic indecomposable $H$-modules is $\{M(k,i)\mid 0\leqslant i\leqslant 3, 1\leqslant k\leqslant2\}$. Let
\begin{gather*}
e_0=\frac{1}{4}(1+b)(1+c),\qquad e_1=\frac{1}{4}(1+b)(1-c),\\
e_2=\frac{1}{4}(1-b)(1+c), \qquad e_3=\frac{1}{4}(1-b)(1-c).
\end{gather*}
It is clear that $\{e_i\mid 0\leqslant i\leqslant 3\}$ is a complete set of primitive orthogonal idempotents of $\Bbbk K_4$ and $e_s\cdot v_i=\delta_{s,i}v_i$. By Lemma \ref{101} and direct calculations, we have that
\begin{gather*}
e_0z=ze_{\tau(0)}=ze_3,\quad e_1z=ze_{\tau(1)}=ze_2,\quad e_2z=ze_{\tau(2)}=ze_1,\quad e_3z=ze_{\tau(3)}=ze_0.
\end{gather*}
Thus, it follows that $\tau(0)=3, \tau(1)=2, \tau(2)=1$ and $\tau(3)=0$.

\begin{lemma}\label{5.1} We have
\begin{gather*}
z\in (e_0+e_1)=(1+b), \quad z\in (e_0+e_2)=(1+c),\\
z\in (e_1+e_3)=(1-c), \quad z\in (e_2+e_3)=(1-b).
\end{gather*}
\end{lemma}

\begin{proof}
Note that $e_0+e_1=\frac{1}{2}(1+b)$.  It follows $(e_0+e_1)=(1+b)$. Since
$$z(1+b)-(1+b)z=2zb\in (1+b)$$
and $b$ is invertible, we obtain $z\in (1+b)$. The proof of the other three formulas is similar.
\end{proof}
We calculate the annihilator ideals of indecomposable $H$-modules in the following corollary.
\begin{corollary}\label{5.2}
The annihilator ideal of $M(2,i)$ is $(ze_{\tau(i)}+e_k+e_{\tau(k)})$ for $0\leqslant i\leqslant 3$ and $k\notin\{i,\tau(i)\}$.
\end{corollary}
\begin{proof}
By Theorem \ref{4.2}, the annihilator ideal of $M(2,i)$ is $(z(1-e_i),\, (1-e_i)(1-e_{\tau(i)}))$. It is clear
$$(z(1-e_i),\, (1-e_i)(1-e_{\tau(i)}))=(z(e_k+e_{\tau(k)}+e_{\tau(i)}),\, e_k+e_{\tau(k)})=(ze_{\tau(i)}+e_k+e_{\tau(k)}).$$
\end{proof}

\begin{corollary}\label{5.3}
The maximal ideals of $H$ are
\begin{align*}
&(e_0+e_1+e_2)=(1+b,1+c),\quad (e_0+e_1+e_3)=(1+b,1-c),\\
&(e_0+e_2+e_3)=(1-b,1+c),\quad (e_1+e_2+e_3)=(1-b,1-c).
\end{align*}
\end{corollary}
\begin{proof}
It is clear by Corollary \ref{4.4} and Lemma \ref{5.1}.
\end{proof}

\begin{lemma}\label{5.4}
If
\begin{equation*}
\bigg(\sum_{i=0}^3k_ize_i+l_ie_i\bigg)=\bigg(\sum_{i=0}^{3}d_ize_i+r_ie_i\bigg)\quad \text{for}\,\, k_i,l_i,d_i,r_i\in \{0,1\},
\end{equation*}
then we have $l_i=r_i$ for $0\leqslant i\leqslant3.$
\end{lemma}

\begin{proof}
If $l_0+l_1+l_2+l_3\neq0$ and $r_0+r_1+r_2+r_3\neq0$, then by Corollary \ref{3.7} we have $l_i=r_i$ for $0\leqslant i\leqslant 3$. If $l_i=0$ for $0\leqslant i\leqslant 3$, then each element in $(\sum_{i=0}^3k_ize_i+l_ie_i)$ is of the form $zv$, $v\in H$. Thus $e_i\notin (\sum_{i=0}^3k_ize_i+l_ie_i)$ for $0\leqslant i\leqslant 3$. Hence $r_i=0$ for $0\leqslant i\leqslant 3$. If $r_i=0$ for $0\leqslant i\leqslant 3$, we can prove $l_i=0$ for $0\leqslant i\leqslant 3$ in a similar way. Therefore we finish the proof.
\end{proof}

\begin{lemma}\label{10}
For $k\notin \{i,\tau(i)\}$, we have $(z(e_k+e_{\tau(k)})+e_i)=(z+e_i).$
\end{lemma}
\begin{proof}
It is clear $(z(e_k+e_{\tau(k)})+e_i)\subseteq(z+e_i)$. We only need to show
$$z\in (z(e_k+e_{\tau(k)})+e_i).$$
Through direct calculations we can obtain
$$z=z(z(e_k+e_{\tau(k)})+e_i)+(z(e_k+e_{\tau(k)})+e_i)(ze_{\tau(i)}+e_k+e_{\tau(k)}).$$
Hence we finish the proof.
\end{proof}

\begin{lemma}\label{5.5}
For $k\notin \{i,\,\tau(i)\}$, the ideals
\begin{gather*}
(e_i),\quad (z+e_i),\quad (ze_k+e_i) \quad\text{and}\quad (ze_{\tau(k)}+e_i)
\end{gather*}
are all different.
\end{lemma}
\begin{proof}
It is easy to see
\begin{gather*}
(e_i)\subseteq (ze_k+e_i)\subseteq(z+e_i)\quad\text{and}\quad (e_i)\subseteq(ze_{\tau(k)}+e_i)\subseteq(z+e_i).
\end{gather*}
We only need to prove
\begin{gather*}
z\notin (ze_{\tau^l(k)}+e_i), \quad ze_{\tau^l(k)}\notin (e_i) \quad\text{and}\quad ze_k\notin (ze_{\tau(k)}+e_i),
\end{gather*}
for $l\in \{0,1\}$. If $z\in (ze_{\tau^l(k)}+e_i)$, then we may assume that
$$z=\sum_s\sum_{t=0}^3(m_{ts}ze_t+n_{ts}e_t)(ze_{\tau^l(k)}+e_i)\sum_{a=0}^3(m_{as}^{\prime}ze_a+n_{as}^{\prime}e_a).$$
Since $\{z^me_n\,|\,0\leqslant m\leqslant 1, \,0\leqslant n\leqslant 3\}$ is a $\Bbbk$-basis of $H$, we have that
\begin{align*}
1&=\sum_s\sum_{t=0}^3\sum_{a=0}^3n_{ts}e_{\tau(t)}e_{\tau^l(k)}n_{as}^{\prime}e_a+m_{ts}e_te_in_{as}^{\prime}e_a+n_{ts}e_{\tau(t)}e_{\tau(i)}m_{as}^{\prime}e_a\\
&=\sum_sn_{\tau^{l-1}(k)s}n_{\tau^l(k)s}^{\prime}e_{\tau^l(k)}+m_{is}n_{is}^{\prime}e_i+n_{is}m_{\tau(i)s}^{\prime}e_{\tau(i)},
\end{align*}
which is in contradiction with $1=e_i+e_{\tau(i)}+e_k+e_{\tau(k)}$. Hence $(z+e_i)\neq(ze_{\tau^l(k)}+e_i)$ for $l\in \{0,1\}$. If $ze_{\tau^l(k)}\in (e_i)$, then we may assume that
$$ze_{\tau^l(k)}=\sum_s\sum_{t=0}^3(u_{ts}ze_t+v_{ts}e_t)e_i\sum_{a=0}^3(u_{as}^{\prime}ze_a+v_{as}^{\prime}e_a).$$
Noting that $\{z^me_n\,|\,0\leqslant m\leqslant 1,\, 0\leqslant n\leqslant 3\}$ is a $\Bbbk$-basis of $H$, we obtain that
\begin{align*}
e_{\tau^l(k)}&=\sum_s\sum_{t=0}^3\sum_{a=0}^3u_{ts}e_te_iv_{as}^{\prime}e_a+v_{ts}e_{\tau(t)}e_{\tau(i)}u_{as}^{\prime}e_a\\
&=\sum_su_{is}v_{is}^{\prime}e_i+v_{is}u_{\tau(i)s}^{\prime}e_{\tau(i)}.
\end{align*}
It is a contradiction of $\tau^l(k)\notin \{i, \tau(i)\}$. Thus, $(ze_{\tau^l(k)}+e_i)\neq(e_i)$ for $l\in \{0,1\}$. Similarly, we can prove $ze_k\notin (ze_{\tau(k)}+e_i)$.
\end{proof}

\begin{lemma}\label{5.6}
For $k\notin \{i,\,\tau(i)\}$, the ideals
\begin{gather*}
(e_i+e_{\tau(i)}), \quad (z+e_i+e_{\tau(i)}), \quad (ze_k+e_i+e_{\tau(i)}) \quad\text{and}\quad (ze_{\tau(k)}+e_i+e_{\tau(i)})
\end{gather*}
are all different.
\end{lemma}

\begin{proof}
It is easy to see that
\begin{align*}
&(e_i+e_{\tau(i)})\subseteq(ze_k+e_i+e_{\tau(i)})\subseteq (z+e_i+e_{\tau(i)})\\
&(e_i+e_{\tau(i)})\subseteq(ze_{\tau(k)}+e_i+e_{\tau(i)})\subseteq (z+e_i+e_{\tau(i)}).
\end{align*}
For $l\in \{0,1\}$, we only need to show that
\begin{gather*}
z\notin (ze_{\tau^l(k)}+e_i+e_{\tau(i)}), \quad  ze_{\tau^l(k)}\notin (e_i+e_{\tau(i)}) \quad\text{and}\quad ze_k\notin (ze_{\tau(k)}+e_i+e_{\tau(i)}).
\end{gather*}
If $z\in(ze_{\tau^l(k)}+e_i+e_{\tau(i)})$, then we may assume
$$z=\sum_s\sum_{t=0}^3(m_{ts}ze_t+n_{ts}e_t)(ze_{\tau^l(k)}+e_i+e_{\tau(i)})\sum_{a=0}^3(m_{as}^{\prime}ze_a+n_{as}^{\prime}e_a).$$
Since $\{z^me_n\,|\,0\leqslant m\leqslant 1,\, 0\leqslant n\leqslant 3\}$ is a $\Bbbk$-basis of $H$, we have that
\begin{align*}
1&=\sum_s\sum_{t=0}^3\sum_{a=0}^3n_{ts}e_{\tau(t)}e_{\tau^l(k)}n_{as}^{\prime}e_a+m_{ts}e_t(e_i+e_{\tau(i)})n_{as}^{\prime}e_a+n_{ts}e_{\tau(t)}(e_{\tau(i)}+e_i)m_{as}^{\prime}e_a\\
&=\sum_sn_{\tau^{l-1}(k)s}n_{\tau^l(k)s}^{\prime}e_{\tau^l(k)}+m_{is}n_{is}^{\prime}e_i+m_{\tau(i)s}n_{\tau(i)s}^{\prime}e_{\tau(i)} +n_{is}m_{\tau(i)s}^{\prime}e_{\tau(i)}+n_{\tau(i)s}m_{is}^{\prime}e_i,
\end{align*}
which is a contradiction of $1=e_i+e_{\tau(i)}+e_k+e_{\tau(k)}$. Thus, 
$$(z+e_i+e_{\tau(i)})\neq(ze_{\tau^l(k)}+e_i+e_{\tau(i)}).$$
If $ze_{\tau^l(k)}\in (e_i+e_{\tau(i)})$, then we may assume
$$ze_{\tau^l(k)}=\sum_s\sum_{t=0}^3(u_{ts}ze_t+v_{ts}e_t)(e_i+e_{\tau(i)})\sum_{a=0}^3(u_{as}^{\prime}ze_a+v_{as}^{\prime}e_a).$$
Note that $\{z^me_n\mid 0\leqslant m\leqslant 1,\, 0\leqslant n\leqslant 3\}$ is a $\Bbbk$-basis of $H$. It follows that
\begin{align*}
e_{\tau^l(k)}&=\sum_s\sum_{t=0}^3\sum_{a=0}^3u_{ts}e_t(e_i+e_{\tau(i)})v_{as}^{\prime}e_a+v_{ts}e_{\tau(t)}(e_{\tau(i)}+e_i)u_{as}^{\prime}e_a\\
&=\sum_su_{is}v_{is}^{\prime}e_i+u_{\tau(i)s}v_{\tau(i)s}^{\prime}e_{\tau(i)}+v_{is}u_{\tau(i)s}^{\prime}e_{\tau(i)}+v_{\tau(i)s}u_{is}^{\prime}e_i,
\end{align*}
which is in contradiction with $\tau^l(k)\notin\{i, \tau(i)\}$. Hence $(ze_{\tau^l(k)}+e_i+e_{\tau(i)})\neq(e_i+e_{\tau(i)})$ for $l\in \{0,1\}$. Similarly, we can prove $ze_k\notin (ze_{\tau(k)}+e_i+e_{\tau(i)})$.
\end{proof}

\begin{theorem}\label{5.7}
$H$ has the following $49$ ideals:
\begin{gather*}
(0),\quad (1),\quad (e_0),\quad (e_1),\quad (e_2),\quad (e_3), \quad (e_0+e_1),\\
(e_0+e_2),\quad (e_0+e_3),\quad (e_1+e_2),\quad (e_1+e_3), \quad (e_2+e_3),\\
(e_0+e_1+e_2),\quad (e_0+e_1+e_3), \quad (e_0+e_2+e_3), \quad (e_1+e_2+e_3),\\
(z),\quad (ze_0),\quad (ze_1),\quad (ze_2),\quad (ze_3), \quad (z(e_0+e_1)),\quad (z(e_0+e_2)),\\
(z(e_0+e_3)),\quad (z(e_1+e_2)),\quad (z(e_1+e_3)), \quad (z(e_2+e_3)),\quad (z(e_0+e_1+e_2)),\\
(z(e_0+e_1+e_3)), \quad (z(e_0+e_2+e_3)), \quad (z(e_1+e_2+e_3)),\\
(z+e_0),\quad (ze_1+e_0),\quad (ze_2+e_0), \quad (z+e_1),\quad (ze_0+e_1),\quad (ze_3+e_1),\\
(z+e_2),\quad (ze_0+e_2),\quad (ze_3+e_2),\quad (z+e_3),\quad (ze_1+e_3),\quad(ze_2+e_3),\\
(z+e_0+e_3),\quad (ze_1+e_0+e_3),\quad (ze_2+e_0+e_3),\\
(z+e_1+e_2),\quad (ze_0+e_1+e_2),\quad (ze_3+e_1+e_2).
\end{gather*}
\end{theorem}

\begin{proof}
We divide the $49$ ideals into three classes:
\begin{align*}
\uppercase\expandafter{\romannumeral1}:\,\, &(0),\quad (1),\quad (e_0),\quad (e_1),\quad (e_2),\quad (e_3), \quad (e_0+e_1),\\
&(e_0+e_2),\quad (e_0+e_3),\quad (e_1+e_2),\quad (e_1+e_3), \quad (e_2+e_3),\\
&(e_0+e_1+e_2),\quad (e_0+e_1+e_3), \quad (e_0+e_2+e_3), \quad (e_1+e_2+e_3),\\
\uppercase\expandafter{\romannumeral2}:\,\, &(z),\quad (ze_0),\quad (ze_1),\quad (ze_2),\quad (ze_3), \quad (z(e_0+e_1)),\\
&(z(e_0+e_2)),\quad (z(e_0+e_3)),\quad (z(e_1+e_2)),\quad (z(e_1+e_3)), \quad (z(e_2+e_3)),\\
&(z(e_0+e_1+e_2)),\quad (z(e_0+e_1+e_3)), \quad (z(e_0+e_2+e_3)), \quad (z(e_1+e_2+e_3)),\\
\uppercase\expandafter{\romannumeral3}:\,\, &(z+e_0),\quad (ze_1+e_0),\quad (ze_2+e_0), \quad (z+e_1),\quad (ze_0+e_1),\quad (ze_3+e_1),\\
&(z+e_2),\quad (ze_0+e_2),\quad (ze_3+e_2),\quad (z+e_3),\quad (ze_1+e_3),\quad (ze_2+e_3),\\
&(z+e_0+e_3),\quad (ze_1+e_0+e_3),\quad (ze_2+e_0+e_3),\quad (z+e_1+e_2),\\
&(ze_0+e_1+e_2),\quad (ze_3+e_1+e_2).
\end{align*}
By Lemma \ref{3.4}, it follows that the ideals in class \uppercase\expandafter{\romannumeral1} are different from the ones in class \uppercase\expandafter{\romannumeral2}. By Lemma \ref{5.4}, it follows that the ideals in class \uppercase\expandafter{\romannumeral2} are different from the ones in class \uppercase\expandafter{\romannumeral3}. By Lemmas \ref{5.4}, \ref{5.5} and \ref{5.6}, we have that the ideals in class \uppercase\expandafter{\romannumeral1} are different from the ones in class \uppercase\expandafter{\romannumeral3}. By Corollary \ref{3.8}, we have that the ideals in class \uppercase\expandafter{\romannumeral1} (or class \uppercase\expandafter{\romannumeral2}) are different. By Lemmas \ref{5.4}, \ref{5.5} and \ref{5.6}, it follows that the ideals in class \uppercase\expandafter{\romannumeral3} are different. Hence the $49$ ideals are all different. In the sequel we only need to show that each ideal of $H$ has to be one of the $49$ ideals. Let $I$ be any non-zero ideal of $H$. Then $I=(h_1),\,(zh_2)$ or $(zh_3+h_4)$, where $h_i\in Z(\Bbbk G)$, $h_i^2=h_i$ for $1\leqslant i\leqslant 4$ and $\langle h_3\rangle\supsetneq \langle h_4\rangle$.
\begin{itemize}
\item When $I=(h_1)$, it follows that $I$ is in class \uppercase\expandafter{\romannumeral1}.
\item When $I=(zh_2)$, it follows that $I$ is in class \uppercase\expandafter{\romannumeral2}.
\item When $I=(zh_3+h_4)$ and $\langle h_3\rangle\supsetneq \langle h_4\rangle$. If $h_4=e_i+e_{\tau(i)}+e_k$ or $e_i+e_k$, $k\notin \{i,\tau(i)\}$, then by Lemma \ref{5.1} we have $z\in (h_4)$. Hence $I=(h_4)$, which is in class \uppercase\expandafter{\romannumeral1}. If $h_4=e_i+e_{\tau(i)}$, noting that $\langle h_3\rangle\supsetneq \langle h_4\rangle$, then we have $h_3=1$ or $e_i+e_{\tau(i)}+e_k$ for $k\notin \{i,\tau(i)\}$. In this case, $I=(z+e_i+e_{\tau(i)})$ or $$I=(z(e_i+e_{\tau(i)}+e_k)+e_i+e_{\tau(i)})=(ze_k+e_i+e_{\tau(i)}),$$
    which is in class \uppercase\expandafter{\romannumeral3}. If $h_4=e_i$, noticing that $\langle h_3\rangle\supsetneq \langle h_4\rangle$, then
    \begin{equation*}
    h_3=e_i+e_{\tau(i)},\,\,e_i+e_k,\,\,e_i+e_{\tau(i)}+e_k,\,\,e_i+e_k+e_{\tau(k)}\,\,\text{or}\,\,1,
    \end{equation*}
    where $k\notin \{i,\tau(i)\}$. By Lemma \ref{10}, we have
     \begin{align*}
    &(z(e_i+e_{\tau(i)})+e_i)=(e_i), \qquad\qquad &(z(e_i+e_{\tau(i)}+e_k)+e_i)=(ze_k+e_i),\\
    &(z(e_i+e_k)+e_i)=(ze_k+e_i),    &(z(e_i+e_k+e_{\tau(k)})+e_i)=(z+e_i).
    \end{align*}
    It follows that $I$ is in class \uppercase\expandafter{\romannumeral1} or class \uppercase\expandafter{\romannumeral3}.
\end{itemize}
Thus, we finish the proof.
\end{proof}

\noindent$\mathbf{Acknowledgements}$ \quad This work was supported by the National Natural Science Foundation of China (Grant No. 11871063).


\begin{thebibliography}{}


\normalsize
\baselineskip=17pt

\bibitem{MR3928507}

A. Alvarado-Garc\'{i}a, C. Cejudo-Castilla, I. F. Vilchis-Montalvo, Parainjectivity, paraprojectivity and artinian principal ideal rings, \emph{J. Algebra Appl.} \textbf{18} (2019) 1950063, 8pp.


\bibitem{MR2100354}
I. M. Ara\'{u}jo, A. V. Kelarev, A. Solomon, An algorithm for commutative semigroup algebras which are principal ideal rings, \emph{Comm. Algebra} \textbf{32} (2004) 1237--1254.
%

\bibitem{MR2197389}
I. Assem, D. Simson, A. Skowro\'nski, \emph{Elements of The Representation Theory of Associative Algebras}. Vol. $1$. Techniques of representation theory, London Mathematical Society Student Texts, 65. Cambridge University Press, Cambridge, 2006.
%
%

\bibitem{MR1314422}
M. Auslander, I. Reiten, S. O. Smal{\o}, \emph{Representation Theory of Artin Algebras}, Cambridge Studies in Advanced Mathematics, 36. Cambridge University Press, Cambridge, 1995.
%


%
%
\bibitem{MR1614186}
S. Catoiu, Ideals of the enveloping algebra $U(sl_2)$, \emph{J. Algebra} \textbf{202} (1998) 142--177.
%


\bibitem{MR2047446}
X. Chen, H. Huang, Y. Ye, P. Zhang, Monomial Hopf algebras, \emph{J. Algebra} \textbf{275} (2004) 212--232.
%


\bibitem{MR3775984}
H. Chimal-Dzul, C. A.  L\'{o}pez-Andrade, When is $R[x]$ a principal ideal ring?, \emph{Rev. Integr. Temas Mat.} \textbf{35} (2017) 143--148.

%
\bibitem{MR1314687}
F. Decruyenaere, E. Jespers, Graded commutative principal ideal rings, \emph{Bull. Soc. Math. Belg. S\'{e}r. B} \textbf{43} (1991) 143--150.
%
\bibitem{MR1125068}
F. Decruyenaere, E. Jespers, P. Wauters, On commutative principal ideal semigroup rings, \emph{Semigroup Forum} \textbf{43} (1991) 367--377.

%
%

\bibitem{MR0396653}
J. L. Fisher, S. K. Sehgal, Principal ideal group rings, \emph{Comm. Algebra} \textbf{4} (1976) 319--325.
%
\bibitem{MR1401676}
E. Jespers, J. Okni\'nski, Semigroup algebras that are principal ideal rings, \emph{J. Algebra} \textbf{183} (1996) 837--863.

\bibitem{MR1321145}
C. Kassel, \emph{Quantum Groups}, Graduate Texts in Mathematics, 155, Springer-Verlag, New York, 1995.
%
\bibitem{MR2236601}
L. Krop, D. E. Radford, Finite-dimensional Hopf algebras of rank one in characteristic zero, \emph{J. Algebra} \textbf{302} (2006) 214--230.
%

\bibitem{MR1856919}
L. Li, P. Zhang, Weight property for ideals of $U_q(sl_2)$, \emph{Comm. Algebra} \textbf{29} (2001) 4853--4870.
%

\bibitem{MR1243637}
S. Montgomery, \emph{Hopf algebras and their actions on rings}, CBMS Regional Conference Series in Mathematics, 82, Amer. Math. Soc., Providence, R.I., 1993.
%
\bibitem{MR0049909}
K. Morita, On group rings over a modular field which possess radicals expressible as principal ideals, \emph{Sci. Rep. Tokyo Bunrika Daigaku}, Sect. A. \textbf{4} (1951) 177--194.
%
\bibitem{MR0457540}

D. S. Passman, Observations on group rings, \emph{Comm. Algebra} \textbf{5} (1977) 1119--1162.

\bibitem{MR0674652}

R. S. Pierce, \emph{Associative Algebras}, Springer-Verlag, New York, 1982.

\bibitem{MR3646318}
S. Siciliano, H. Usefi, Enveloping algebras that are principal ideal rings, \emph{J. Pure Appl. Algebra} \textbf{221} (2017) 2573--2581.
%
\bibitem{MR2397414}
S. Scherotzke, Classification of pointed rank one Hopf algebras, \emph{J. Algebra} \textbf{319} (2008) 2889--2912.
%

\bibitem{yuwang}
Y. Wang, Classification of ideals of $8$-dimensional Radford Hopf algebra, \emph{Czech. Math. J.}, published on line, doi:10.21136/CMJ.2022.0313-21.

\bibitem{MR3284336}
Z. Wang, L. Li, Y. Zhang, Green rings of pointed rank one Hopf algebras of nilpotent type, \emph{Algebr. Represent. Theory} \textbf{17} (2014) 1901--1924.
%
\bibitem{MR3448167}
Z. Wang, L. Li, Y. Zhang, Green rings of pointed rank one Hopf algebras of non-nilpotent type, \emph{J. Algebra} \textbf{449} (2016) 108--137.
%



\bibitem{MR4256338}
Y. Wang, Z. Wang, L. Li, Ideals of finite-dimensional pointed Hopf algebras of rank one, \emph{Algebra Colloq.} \textbf{28} (2021) 351--360.
%



\bibitem{MR3357949}
Z. Wang, L. You, H. Chen, Representations of Hopf-Ore extensions of group algebras and pointed Hopf algebras of rank one, \emph{Algebr. Represent. Theory} \textbf{18} (2015) 801--830.
%
%

\bibitem{MR4296825}
Y. Wang, Y. Zheng, L. Li, On the ideals of the Radford Hopf algebras, \emph{Comm. Algebra} \textbf{49} (2021) 4109--4122.
%


\bibliographystyle{plain}
\end{thebibliography}
\end{document}